\documentclass[12pt]{article}

\usepackage[a4paper,margin=1in]{geometry}
\usepackage{amsmath,amsthm,amssymb,mathtools,bm}
\usepackage{hyperref}
\usepackage[T1]{fontenc}
\usepackage{microtype}
\usepackage{comment}
\hypersetup{colorlinks=true,linkcolor=blue,citecolor=blue,urlcolor=blue}

\newtheorem{theorem}{Theorem}[section]
\newtheorem{lemma}[theorem]{Lemma}

\theoremstyle{definition}

\newtheorem{remark}[theorem]{Remark}

\numberwithin{equation}{section}

\allowdisplaybreaks

\usepackage[usenames]{color}

\newcommand{\rd}{\,\mathrm{d}}

\title{The star discrepancy of a union of randomly digitally shifted Korobov polynomial lattice point sets depends polynomially on the dimension}

\author{Josef Dick\thanks{School of Mathematics and Statistics, The University of New South Wales Sydney, 2052 NSW, Australia (\url{josef.dick@unsw.edu.au}). The work of J.~D.\ is supported by ARC grant DP220101811.} \, and 
Friedrich Pillichshammer\thanks{Institute for Financial Mathematics and Applied Number Theory, Johannes Kepler University Linz, Altenbergerstra{\ss}e 69, 4040 Linz, Austria (\url{friedrich.pillichshammer@jku.at})}}
\date{\today}

\begin{document}

\maketitle

\centerline{{\it Dedicated to Gerhard Larcher on the occasion of his retirement.}}

\begin{abstract}
The star discrepancy is a quantitative measure of the uniformity of a point set in the unit cube. A central quantity of interest is the inverse of the star discrepancy, $N(\varepsilon, s)$, defined as the minimum number of points required to achieve a star discrepancy of at most~$\varepsilon$ in dimension~$s$. It is known that $N(\varepsilon, s)$ depends only linearly on the dimension~$s$. 
Finding explicit point set constructions that achieve this optimal linear dependence on the dimension remains a major open problem. 

In this paper, we make progress on this question by analyzing point sets constructed from a multiset union of digitally shifted Korobov polynomial lattice point sets. Specifically, we show the following two results. A union of \emph{randomly generated} Korobov polynomial lattice point sets shifted by a random digital shift of depth $m$ can achieve a star discrepancy whose inverse depends only linearly on $s$. The second result shows that a union of \emph{all} Korobov polynomial lattice point sets, each shifted by a different random digital shift, achieves the same star discrepancy bound. While our proof relies on a concentration result (Bernstein's inequality) and is therefore non-constructive, it significantly reduces the search space for such point sets from a continuum of possibilities to a finite set of candidates, marking a step towards a fully explicit construction.
\end{abstract}

\centerline{\begin{minipage}[hc]{130mm}{
}
\end{minipage}}

\section{Introduction and statement of the results}
Let $s,N \in \mathbb{N}$. For a finite set $P=\{\bm{x}_0,\dots,\bm{x}_{N-1}\}$ in $[0,1)^s$ the \emph{star discrepancy}
\[
  D^*_N(P) \;:=\; \sup_{\bm{z}\in[0,1]^s}\left|\frac{1}{N}\sum_{n=0}^{N-1} \mathbf{1}_{[\bm{0},\bm{z})}(\bm{x}_n)-\lambda([\bm{0},\bm{z}))\right|,
\]
where $\bm{1}_{[\bm{0},\bm{z})}$ denotes the indicator function and $\lambda$ the Lebesgue measure, measures the maximal deviation between the empirical distribution of $P$ and the uniform distribution over axis-parallel boxes anchored at the origin. Its inverse,
\[
  N(\varepsilon,s) \;:=\; \min\left\{N\in\mathbb{N}:\ \exists P\subseteq [0,1)^s,\ |P|=N,\ D^*_N(P)\le\varepsilon\right\},
\]
for $\varepsilon \in (0,1]$ and $s \in \mathbb{N}$ is a central complexity parameter in discrepancy theory and quasi–Monte Carlo (QMC) integration; by the Koksma–Hlawka inequality (and its modern refinements), small discrepancy implies small worst-case integration error for broad function classes~\cite{AD15,DKP10}.

A landmark result due to Heinrich, Novak, Wasilkowski, and Wo\'zniakowski \cite[Theorem~3]{HNWW01} shows that the inverse star discrepancy depends \emph{linearly} on the dimension:
\begin{equation}\label{eq:HNWW}
  N(\varepsilon,s) \;\le\; C\, \frac{s }{\varepsilon^2},
\end{equation}
for a universal constant $C>0$. (The currently smallest known value of $C$ is $C \approx 6.0673$ as shown in \cite{Wei26}.) Equivalently, there exists a $c>0$ such that for any $N,s\in \mathbb{N}$ there exist $N$-point sets in $[0,1)^s$ with $D^*_N(P)\le c\sqrt{s/N}$. The proof is probabilistic: independent uniform sampling combined with a concentration inequality delivers the bound with positive probability. Complementing~\eqref{eq:HNWW}, Heinrich et al.~\cite[Theorem~8]{HNWW01} proved the lower bound $N(\varepsilon,s)\ge c\, s \log \varepsilon^{-1}$ which was improved by   Hinrichs~\cite{Hin04} to $N(\varepsilon,s)\ge c\, s/\varepsilon$ for sufficiently small $\varepsilon$. Both results show, together with the upper bound \eqref{eq:HNWW}, that the inverse star discrepancy is linear in the dimension $s$; elementary arguments leading to this linear-in-$s$ barrier were later developed by Steinerberger~\cite{Steinerberger22}. For a detailed discussion see \cite[Sec.~3.1.5]{NW08} and \cite[Sec.~9.9]{NW10}.  There is also a sharp picture for random points: the expected star discrepancy of an $N$-point i.i.d.\ sample is of order $\sqrt{s/N}$, as shown in~\cite{Doerr14}, and explicit probability-tail bounds are known, see~\cite{AH14} and \cite{GPW}.

Despite the clarity of the existential theory, \emph{explicit} constructions that achieve the optimal linear dependence on~$s$ have remained elusive. Classical low-discrepancy families---digital nets and sequences, rank-1 lattices, polynomial lattices---achieve the best known asymptotic order in $N$ for fixed $s$, namely $D^*_N(P)=\mathcal{O}\bigl(N^{-1}(\log N)^{s-1}\bigr)$, but this deteriorates with growing dimension~\cite{DKP22,DKS13,DKP10,niesiam}. Considerable progress has been made on tractability via \emph{weights}, where coordinate importance decays and dimension dependence improves or even disappears~\cite{Aistleitner2014weights,HPS08}; constructive component-by-component (CBC) and fast CBC algorithms underpin state-of-the-art lattice and polynomial lattice rules in weighted settings~\cite{DKLP07,DKLS05,NC06a,NC06b}. Still, for the \emph{unweighted} star discrepancy, attaining the $\sqrt{s/N}$ order by explicit points remains open.

From a computational perspective, the exact star discrepancy is hard to evaluate: determining $D^*_N(P)$ is NP-hard~\cite{GSW09}. This has motivated randomized and heuristic approximations~\cite{Clément23GECCO,GWW12} and specialized optimization viewpoints (e.g., subset selection)~\cite{Clément22JCO}. These algorithmic challenges amplify the value of structural constructions that narrow the search space. There are also deterministic algorithms based on derandomization that construct $N$ points such that their star discrepancy is below a certain threshold $\varepsilon$ and the number of points grows like $s \log s$ which may be viewed as ``constructive proofs'' of a slightly weaker result; see, e.g., \cite{DGW10}.

\medskip
In this paper we investigate point sets obtained as \emph{multiset unions of multiple digitally shifted Korobov polynomial lattice point sets}. Multiple rank-1 lattice ideas, developed primarily for sparse trigonometric approximation and sampling~\cite{GIWV20,Kaemmerer18,KPV21}, suggest that a small number of carefully chosen periodic structures can mimic randomness while retaining favorable arithmetic properties. 

Upper bounds for numerical integration in subspaces of the Wiener algebra which depend only polynomially on the dimension have been studied in \cite{D14, DGS25, G23}. The construction of the point set in these papers is often explicit. The problem of numerical integration in subspaces of the Wiener algebra is loosely related to the star discrepancy problem; however, the constructions from \cite{D14, DGS25, G23} do not seem to be directly applicable to the inverse of the star discrepancy problem, although the constructions in the present paper have similarities with the constructions there.

We prove that a union of a modest number of independently generated digitally shifted polynomial lattice point sets (where the shift is of depth $m$) achieves
\[
  D^*_N(P) \;\lesssim\; \frac{s \log N}{\sqrt{N}},
\]
with high probability, thus almost matching the optimal order of~\eqref{eq:HNWW}. The argument follows the blueprint from \cite{HNWW01} but with two crucial modifications:
\begin{enumerate}
\item We restrict sampling to a \emph{finite} candidate family consisting of multiset unions of structured point sets (e.g., all unions formed from a set of digitally shifted Korobov polynomial lattices), where the shift is of depth $m$ (this implies that the number of possible shifts is $N^{s}$ where $N = 2^m$). 
\item We employ Bernstein's inequality~\cite{ber1946} to control deviations of suitably aggregated discrepancy contributions across the randomly chosen digitally shifted Korobov polynomial lattice point sets.
\end{enumerate}
Consequently, although our proof remains non-constructive (it shows existence), the ambient search space collapses from a continuum to a \emph{finite} and explicitly parameterized set of candidates. This reduction brings the goal of a fully explicit construction closer: one can hope to derandomize within a concrete finite family, or to certify good instances using improved discrepancy approximations.

\begin{remark}\label{re1.1}
Consider the result from \cite{HNWW01}. Instead of sampling points independently and uniformly from $[0,1]^s$, one may discretize the problem and sample i.i.d.\ from a finite grid with mesh size $1/\sqrt{sN}$. Standard probabilistic arguments show that sampling $N$ points independently from this grid yields the same order of existence result for the star discrepancy as in the theorem of Heinrich et al., namely
\[
D_N^* \le C \frac{s \log N}{\sqrt{N}}.
\]

Such a grid contains $(s N)^{s/2}$ elements. 
Hence, the number of possible $N$-point selections from this grid (allowing repetition and ordered choice) is of order
\[
\bigl((sN)^{s/2}\bigr)^N = (sN)^{sN/2}.
\]
This quantity provides a natural reference scale for the size of the search space associated with any randomized or semi-randomized construction, since in practice any such procedure operates on a discretized (and therefore finite) set of candidates.

For comparison, the search spaces arising in Theorem~\ref{thm1} and Theorem~\ref{thm2} are substantially smaller. A direct count shows that they are of order
$$N^{1+s \sqrt{N}/2} \qquad\text{and}\qquad N^{s \sqrt{N}/2}$$
respectively, which is markedly below $(sN)^{sN/2}$. 
Thus, from the viewpoint of combinatorial complexity, the constructions considered here explore a reduced portion of the fully discretized search space while still achieving the desired discrepancy bounds.
\end{remark}

\medskip
\noindent\textbf{Related work.}
Our use of unions of structured nodes connects to numerical integration in subspaces of the Wiener algebra \cite{D14,DGS25,G23}, to multiple rank-1 lattice sampling for high-dimensional Fourier problems~\cite{Kaemmerer18,KPV21} and to deterministic constructions of multiple lattices with guaranteed reconstruction properties~\cite{GIWV20}. On the QMC side, CBC-type constructions for lattices and polynomial lattices are extensively developed~\cite{DKLP07,DKLS05,NC06a,NC06b}, including higher-order variants~\cite{Baldeaux2011}. A basic tool in our proof is Bernstein's inequality. The use of this inequality for proving probabilistic discrepancy bounds was proposed first by Aistleitner~\cite{Ai11}. Weighted star discrepancy has a rich tractability theory~\cite{Aistleitner2014weights,DKP10,HPS08}, with refinements for specific sequences (e.g., Halton, see~\cite{HinrichsMarkhasinHalton}). At the existential level, our result stands alongside the probabilistic bounds for random and negatively dependent samples~\cite{AH14,DoerrJittered,DGKP08,DGW10,GPW,Lobbe14,NP,WGH20}.

\medskip
\noindent\textbf{Contributions and paper outline.}
In Theorem~\ref{thm1}, we formalize a randomized selection of digitally shifted Korobov polynomial lattice point sets and prove that with high probability a multiset union of independently chosen digitally shifted Korobov polynomial lattice point sets attains $D^*_N(P)\lesssim s \log(N)/\sqrt{N}$ (with an implied constant independent of $s,N$), thus recovering the almost optimal order of the inverse discrepancy from~\cite{HNWW01} with respect to the dimension $s$ within a finite candidate family. 
Section~\ref{sec:main} presents the main probabilistic existence theorem and its proof via Bernstein's inequality. In Theorem~\ref{thm2} we show that we can take the union of all Korobov polynomial lattice rules, each randomly shifted by a different digital shift of depth $m$, to obtain the same bound on the star discrepancy. The proof of this result is in Section~\ref{sec:main2}.

\subsection{Notation}

For $m \in \mathbb{N}$ write $\mathbb{Q}_{2^m}:=\{0,\tfrac{1}{2^m},\tfrac{2}{2^m},\ldots,\tfrac{2^m-1}{2^m}\}$ and $\overline{\mathbb{Q}}_{2^m}:=\mathbb{Q}_{2^m} \cup \{1\}$. 

\medskip

Let $P = \{\bm{x}_0, \bm{x}_1, \ldots, \bm{x}_{N-1} \}$ be a set of $N$ points in the $s$-dimensional unit cube $[0,1)^s$. For an axis-parallel box $J = \prod_{j=1}^s [0, b_j) \subseteq [0,1)^s$, the \textit{local discrepancy} of $P$ in $J$ is given by
\begin{equation}\label{def_Delta}
\Delta(P, J) := \frac{A(J,P)}{N} - \lambda(J),
\end{equation}
where $A(J,P):=\sum_{n=0}^{N-1}\mathbf{1}_J(\bm{x}_n)$ denotes the number of points in $P$ that belong to $J$, and $\lambda(J) = \prod_{j=1}^s b_j$ is the $s$-dimensional Lebesgue measure of $J$.

For our analysis, we consider special dyadic boxes $J(\bm{b}) = [\bm{0}, \bm{b}) = \prod_{j=1}^s [0, b_j)$ with $\bm{b} =(b_1, \ldots, b_s) \in \overline{\mathbb{Q}}_{2^m}^s$. For such $\bm{b}$ let $\{\bm{0}, \ldots, \bm{b} 2^m - \bm{1}\} = \prod_{j=1}^s \{0, 1, \ldots, b_j 2^m-1\}$.

\bigskip

Our construction of point sets is based on the finite field $\mathbb{Z}_2 = \{0, 1\}$ equipped with the usual arithmetic operations modulo $2$. In particular, we denote the addition in $\mathbb{Z}_2$ by $\oplus$. Let $\mathbb{Z}_2[x]$ be the ring of polynomials with coefficients in $\mathbb{Z}_2$, and for a positive integer $m$, let $G_m := \{ q \in \mathbb{Z}_2[x] : \deg(q) < m \}$. We fix an irreducible polynomial $p \in \mathbb{Z}_2[x]$ of degree $m$.

Each non-negative integer $n$ can be identified with a polynomial $n(x) \in \mathbb{Z}_2[x]$ in the natural way via its dyadic expansion. If $n = \sum_{i=0}^r n_i 2^i$ with coefficients $n_i \in \{0,1\}$, we associate the polynomial $n(x) = \sum_{i=0}^r n_i x^i$. Note that the integers $n \in \{0, 1, \ldots, 2^m-1\}$ correspond precisely to the polynomials in $n(x) \in G_m$. We will use $n \in \{0, 1, \ldots, 2^m-1\}$ and $n \in G_m$ interchangeably. It should be clear from the context whether $n$ is to be considered an element in $\{0, 1,\ldots, 2^m-1\}$ or $G_m$.

Points are generated using the field of formal Laurent series $\mathbb{Z}_2((x^{-1}))$, whose elements are of the form $L(x) = \sum_{\ell=w}^{-\infty} a_\ell x^\ell$ for some integer $w$ and $a_\ell \in \mathbb{Z}_2$. Any rational function $g(x)/p(x) \in \mathbb{Z}_2(x)$ has a unique expansion in $\mathbb{Z}_2((x^{-1}))$. We define a map $\nu_m: \mathbb{Z}_2((x^{-1})) \to [0,1)$ that truncates this expansion: for $L(x) = \sum_{\ell=w}^{-\infty} a_\ell x^\ell$, we set
\[
\nu_m(L(x)) = \sum_{\ell=1}^{m} \frac{a_{-\ell}}{2^\ell}.
\]
where we set $a_{\ell} = 0$ for $\ell > w$. Obviously, $\nu_m(L(x))  \in \mathbb{Q}_{2^m}$.

\paragraph{Korobov polynomial lattice point set.}
A \textit{Korobov polynomial lattice point set} $P_p(q)$ with $N=2^m$ points is defined by a generating vector $\bm{q} = (1, q, q^2 \ldots, q^{s-1}) \pmod{p}$, which we view as an element in $G_m^s$, and the modulus $p$. The points $\bm{x}_n(q) = (x_{n,1}(1), x_{n,2}(q), \ldots, x_{n,s}(q^{s-1}))$ are given by
\[
\bm{x}_n(q) = \left( \nu_m\left(\frac{n(x)}{p(x)}\right), \nu_m\left( \frac{n(x) q(x) }{p(x)} \right), \ldots, \nu_m\left(\frac{n(x)q^{s-1}(x)}{p(x)}\right) \right) \quad \text{for } n = 0, 1, \ldots, 2^m-1.
\]
Obviously, $\bm{x}_n(q) \in \mathbb{Q}_{2^m}^s$.

\paragraph{Digital shifts.} Let $\sigma \in \mathbb{Q}_{2^m}$ with dyadic expansion $\sigma=\frac{\varsigma_1}{2}+\cdots+\frac{\varsigma_m}{2^m}$ with digits $\varsigma_1,\ldots,\varsigma_m \in \mathbb{Z}_2$. For $x \in  \mathbb{Q}_{2^m}$ with dyadic expansion $x=\frac{\xi_1}{2}+\cdots+\frac{\xi_m}{2^m}$ with digits $\xi_1,\ldots,\xi_m \in \mathbb{Z}_2$ we define the ($\sigma$-)digitally shifted point $x \oplus \sigma$ via its dyadic expansion as $$x \oplus \sigma:=\frac{\varsigma_1 \oplus \xi_1}{2}+\cdots +\frac{\varsigma_m \oplus \xi_m}{2}.$$ For vectors $\bm{\sigma},\bm{x} \in \mathbb{Q}_{2^m}^s$ the digitally shifted point $\bm{x} \oplus \bm{\sigma}$ is defined component-wise and for a set $P=\{\bm{x}_0,\ldots,\bm{x}_{N-1}\} \subseteq \mathbb{Q}_{2^m}^s$ and $\bm{\sigma} \in \mathbb{Q}_{2^m}^s$ the digitally shifted set is defined as $P\oplus \bm{\sigma}:=\{\bm{x}_0\oplus \bm{\sigma},\ldots,\bm{x}_{N-1}\oplus \bm{\sigma}\}$.

For $\bm{\sigma} \in \mathbb{Q}_{2^m}^s$, we denote by $P_p(q) \oplus \bm{\sigma}$ the digitally shifted Korobov polynomial lattice point set.

\paragraph{Walsh functions.} For $k \in \mathbb{N}_0$ let $k = \kappa_0 + \kappa_1 2 + \cdots + \kappa_{m-1} 2^{m-1}$ be the dyadic expansion of $k$ with digits $\kappa_0, \kappa_1, \ldots, \kappa_{m-1} \in \mathbb{Z}_2$. Similarly, let $x \in [0,1)$ have the dyadic expansion $x = \xi_1 2^{-1} + \xi_2 2^{-2} + \cdots$ with $\xi_1, \xi_2, \ldots \in \mathbb{Z}_2$, assuming infinitely many of the digits $\xi_i$ differ from $1$. For $k \in \mathbb{N}_0$ we define the $k$-th Walsh function $\mathrm{wal}_k: [0,1) \to \{-1,1\}$ by
\begin{equation*}
\mathrm{wal}_k(x) := (-1)^{\xi_1 \kappa_0 + \xi_2 \kappa_1 + \cdots + \xi_{m} \kappa_{m-1}}.
\end{equation*}
Details about Walsh functions can be found in \cite{Fine,SWS} or in \cite[Appendix~A]{DKP10}. For example, we will use that for $x,\sigma \in \mathbb{Q}_{2^m}$ we have $\mathrm{wal}_k(x \oplus \sigma)=\mathrm{wal}_k(x) \mathrm{wal}_k(\sigma)$, see \cite[Appendix~A]{DKP10}.

\subsection{Main results}\label{sec_main_results}

Our main results are as follows.

\begin{theorem}\label{thm1}
Let $m\in\mathbb{N}$ and let $p\in\mathbb{Z}_2[x]$ be an irreducible polynomial of degree $m$. 
Let $q_1,\dots,q_{2^m}\in G_m$ and $\boldsymbol{\sigma}_1,\dots,\boldsymbol{\sigma}_{2^m}\in \mathbb{Q}_{2^m}^s$ be independent and uniformly distributed. Define the multiset
\begin{equation}\label{defP}
  P \;:=\; \bigcup_{r=1}^{2^m} \bigl( P_p(q_r)\,\oplus\,\boldsymbol{\sigma}_r \bigr),
\end{equation}
and write $N:=|P|=2^{2m}$. Then, for every $\delta\in(0,1)$, with probability at least $\delta$, the star discrepancy of $P\subseteq \mathbb{Q}_{2^m}^s$ satisfies
\[
  D_N^\ast(P)\;\le\; (1.723\ldots) \times  \frac{s (\log(2 N) +1)+ \log 2 - \log(1-\delta)}{N^{1/2}}.
\]
\end{theorem}

The proof of this result is in Section~\ref{sec:main}.

\begin{theorem}\label{thm2}
Let $m\in\mathbb{N}$ and let $p\in\mathbb{Z}_2[x]$ be an irreducible polynomial of degree $m$. 
Let  $\boldsymbol{\sigma}_1,\dots,\boldsymbol{\sigma}_{2^m}\in \mathbb{Q}_{2^m}^s$ be independent and uniformly distributed. Define the multiset
\begin{equation}\label{defP2}
  Q \;:=\; \bigcup_{r=0}^{2^m-1} \bigl( P_p(r)\,\oplus\,\boldsymbol{\sigma}_{r+1} \bigr),
\end{equation}
and write $N:=|Q|=2^{2m}$. Then, for every $\delta\in(0,1)$, with probability at least $\delta$, the star discrepancy of $Q\subseteq \mathbb{Q}_{2^m}^s$ satisfies
\[
  D_N^\ast(Q)\;\le\; (1.723\ldots) \times  \frac{s (\log(2 N) +1)+ \log 2 - \log(1-\delta)}{N^{1/2}}.
\]
\end{theorem}

In contrast to Theorem~\ref{thm1}, the construction here eliminates the random choice of polynomials. The point set $Q$ is a variation of a Korobov $p$-set (see, e.g., \cite[Chapter~6]{DKP22}), where each polynomial Korobov lattice point set is randomly shifted by an i.i.d. random digital shift of depth $m$. The proof of this result is in Section~\ref{sec:main2}.

\section{Proofs}

Before presenting the proofs of Theorem~\ref{thm1} and \ref{thm2}, we need some auxilliary results.

\subsection{Auxilliary results}

We represented the indicator function using a Walsh series. The result is well-known in a more general context (see, e.g., \cite{DKP10,Fine}). Since it is more complex to derive the special case we consider from these general results than to show the result directly, we include a short proof for simplicity.

\begin{lemma}
For $b \in \overline{\mathbb{Q}}_{2^m}$ we have 
\begin{align}
\bm{1}_{[0, b )}(x) & = \sum_{k=0}^{2^{m}-1} c_k \operatorname{wal}_k(x),\label{fine3}
\end{align}
where $$c_k = c_k(b) := \frac{1}{2^{m}} \sum_{v = 0}^{b 2^{m} - 1} \operatorname{wal}_k\left(\frac{v}{2^m}\right) \qquad \text{(note that $c_0(b)=b$).}$$ 
\end{lemma}

\begin{proof}
It suffices the calculate the Walsh coefficients $c_k$ of $\bm{1}_{[0, b )}$ for $b \in \overline{\mathbb{Q}}_{2^m}$. We have
$$c_k=\int_0^b \operatorname{wal}_k(x) \rd x=\sum_{v=0}^{b 2^m-1} \int_{v/2^m}^{(v+1)/2^m}\operatorname{wal}_k(x) \rd x.$$ Let $v=v_0+v_1 2+\cdots+v_{m-1} 2^{m-1}$ with dyadic digits $v_0,v_1\ldots,v_{m-1}\in \{0,1\}$. Then $x \in [\tfrac{v}{2^m},\tfrac{v+1}{2^m})$ has dyadic expansion $$x=\frac{v_{m-1}}{2}+\cdots+\frac{v_0}{2^m}+\frac{\xi_{m+1}}{2^{m+1}}+\frac{\xi_{m+2}}{2^{m+2}}+\cdots$$ with dyadic digits $\xi_{m+1},\xi_{m+2},\ldots \in \{0,1\}$. Let $k=\kappa_0+\kappa_1 2+\cdots$ with dyadic digits $\kappa_0,\kappa_1,\ldots \in \{0,1\}$ (which eventually  become 0). If $k\ge 2^m$, then there exists an index $j \ge m$ with $\kappa_j=1$ and hence $$\int_{v/2^m}^{(v+1)/2^m}\operatorname{wal}_k(x) \rd x = (-1)^{\kappa_0 v_{m-1}+\ldots+\kappa_{m-1} v_0} \int_{v/2^m}^{(v+1)/2^m} (-1)^{\kappa_m \xi_{m+1}+\cdots} \rd x=0.$$ So in this case we have $c_k=0$. If $k < 2^m$, then $\kappa_j=0$ for all $j \ge m$ and hence $$\int_{v/2^m}^{(v+1)/2^m}\operatorname{wal}_k(x) \rd x = (-1)^{\kappa_0 v_{m-1}+\ldots+\kappa_{m-1} v_0} \int_{v/2^m}^{(v+1)/2^m} \rd x=\frac{1}{2^m} \operatorname{wal}_k\left(\frac{v}{2^m}\right).$$ Thus $$c_k=\frac{1}{2^m} \sum_{v=0}^{b 2^m-1} \operatorname{wal}_k\left(\frac{v}{2^m}\right).$$
\end{proof}

For $\bm{k} = (k_1, k_2, \ldots, k_s) \in \{0,1,\ldots,2^m-1\}^s$ and for $\bm{b}=(b_1,\ldots,b_s) \in \overline{\mathbb{Q}}_{2^m}^s$, let 
\begin{equation*}
c_{\bm{k}} = c_{\bm{k}}(\bm{b}) := \prod_{j=1}^s c_{k_j}(b_j) = \frac{1}{2^{ms}} \sum_{\bm{v} \in \{\bm{0}, \ldots, \bm{b} 2^m - \bm{1}\}} \operatorname{wal}_{\bm{k}}(\bm{v} 2^{-m}).
\end{equation*}
In particular, we have $c_{\bm{0}} = c_{\bm{0}}(\bm{b}) = \prod_{j=1}^s b_j$.

The following lemma shows that random digital shifts $\bm{\sigma} \in \mathbb{Q}_{2^m}^s$ ensure that the expected value of the discrepancy function is $0$ for all $J(\bm{b})$ with $\bm{b} \in \overline{\mathbb{Q}}_{2^m}^s$.

\begin{lemma}\label{lem_Q}
Let $E \subseteq [0,1)^s$ be an arbitrary point set. Let $\bm{\sigma} \in \mathbb{Q}_{2^m}^s$ be chosen uniformly distributed. Then for any $\bm{b}=(b_1,\ldots,b_s) \in \overline{\mathbb{Q}}_{2^m}^s$  we have $\mathbb{E}_{\bm{\sigma}} [\Delta(E \oplus \bm{\sigma},J(\bm{b})) ] = 0$.
\end{lemma}

\begin{proof}
Using \eqref{fine3}, for $\bm{\sigma} \in \mathbb{Q}_{2^m}^s$ we have
\begin{equation}\label{fo:Delta1}
\mathbb{E}_{\bm{\sigma}} [\Delta(E \oplus \bm{\sigma}, J(\bm{b})) ] = \sum_{\bm{k} \in \{0, \ldots, 2^{m}-1\}^s \setminus \{\bm{0}\}}  c_{\bm{k}} \frac{1}{|E|} \sum_{\bm{x} \in E} \mathbb{E}_{\bm{\sigma}} [ \operatorname{wal}_{\bm{k}}(\bm{x} \oplus \bm{\sigma}) ],
\end{equation}
and hence 
\begin{align*}
\mathbb{E}_{\bm{\sigma}} [\Delta(E \oplus \bm{\sigma}, J(\bm{b})) ] & = \sum_{\bm{k} \in \{0, \ldots, 2^m-1\}^s \setminus \{\bm{0}\}} c_{\bm{k}} \frac{1}{|E|} \sum_{\bm{x} \in E} \frac{1}{2^{ ms}} \sum_{\bm{\sigma} \in \mathbb{Q}_{2^m}^s} \operatorname{wal}_{\bm{k}}(\bm{x} \oplus \bm{\sigma}) \\ & = \sum_{\bm{k} \in \{0, \ldots, 2^m-1\}^s \setminus \{\bm{0}\}} c_{\bm{k}} \frac{1}{|E|} \sum_{\bm{x} \in E} \operatorname{wal}_{\bm{k}}(\bm{x}) \frac{1}{2^{ ms}} \sum_{\bm{\sigma} \in \mathbb{Q}_{2^m}^s} \operatorname{wal}_{\bm{k}}(\bm{\sigma}).
\end{align*}
Since
\begin{equation*}
\frac{1}{2^{ms}} \sum_{\bm{\sigma} \in \mathbb{Q}_{2^m}^s} \operatorname{wal}_{\bm{k}}(\bm{\sigma}) = \begin{cases} 0 & \mbox{for } \bm{k} \neq \bm{0}, \\ 1 & \mbox{for } \bm{k} = \bm{0}, \end{cases}
\end{equation*}
we obtain $\mathbb{E}_{\bm{\sigma}} [\Delta(E \oplus \bm{\sigma},J(\bm{b})) ] = 0$.
\end{proof}

\begin{lemma}\label{lem_aux}
For $\bm{b}=(b_1,\ldots,b_s) \in \overline{\mathbb{Q}}_{2^m}^s$ we have
\begin{align*}
\sum_{\bm{k} \in \{0, \ldots, 2^m-1\}^s \setminus\{\bm{0}\}} c_{\bm{k}}(\bm{b}) & = 1 - \prod_{j=1}^s b_j, \\
\sum_{\bm{k} \in \{0, \ldots, 2^m-1\}^s \setminus\{\bm{0}\}} c_{\bm{k}}(\bm{b})^2 & = \prod_{j=1}^s b_j \left( 1 - \prod_{j=1}^s b_j \right).
\end{align*}
\end{lemma}

\begin{proof}
We have $$\sum_{\bm{k} \in \{0, \ldots, 2^m-1\}^s \setminus\{\bm{0}\}} c_{\bm{k}}(\bm{b}) = \sum_{\bm{k} \in \{0, \ldots, 2^m-1\}^s} c_{\bm{k}}(\bm{b}) -c_{\bm{0}}(\bm{b}) = \prod_{j=1}^s \left(\sum_{k=0}^{2^m-1} c_k(b_j)\right) - \prod_{j=1}^s b_j.$$
From here the first result follows, because $$\sum_{k=0}^{2^m-1} c_k(b)=\frac{1}{2^m} \sum_{k=0}^{2^m-1} \sum_{v=0}^{2^m b-1}  \operatorname{wal}_k(\tfrac{v}{2^m})=1+\frac{1}{2^m}\sum_{v=1}^{2^m b -1} \underbrace{\sum_{k=0}^{2^m-1}\operatorname{wal}_k(\tfrac{v}{2^m})}_{=0}=1.$$

In the same way we have
$$\sum_{\bm{k} \in \{0, \ldots, 2^m-1\}^s \setminus\{\bm{0}\}} c_{\bm{k}}(\bm{b})^2 = \sum_{\bm{k} \in \{0, \ldots, 2^m-1\}^s} c_{\bm{k}}(\bm{b})^2 -c_{\bm{0}}(\bm{b})^2 = \prod_{j=1}^s \left(\sum_{k=0}^{2^m-1} c_k(b_j)^2\right) - \prod_{j=1}^s b_j^2.$$ From here the second result follows, because 
\begin{align*}
\sum_{k=0}^{2^m-1} c_k(b)^2 = & \frac{1}{2^{2m}} \sum_{k=0}^{2^m-1} \sum_{v,v'=0}^{2^m b-1}  \operatorname{wal}_k(\tfrac{v}{2^m} \oplus \tfrac{v'}{2^m})\\
= & \frac{1}{2^{2m}} \sum_{k=0}^{2^m-1} \underbrace{\sum_{v=0}^{2^m b-1}  \operatorname{wal}_k(0)}_{=b 2^{m}}+\frac{1}{2^{2m}}  \sum_{v,v'=0 \atop v\not= v'}^{2^m b-1}  \underbrace{\sum_{k=0}^{2^m-1}\operatorname{wal}_k(\tfrac{v}{2^m} \oplus \tfrac{v'}{2^m})}_{=0}=b.
\end{align*}

\end{proof}

\subsection{Star-discrepancy estimate}

Let $E_1, \ldots, E_{2^m} \subseteq [0,1)^s$, with $E_r = \{\bm{x}_0(r), \ldots, \bm{x}_{2^m-1}(r)\}$ for $ r = 1, \ldots, 2^m$, be arbitrary point sets. Define the multiset union $E = \bigcup_{r=1}^{2^m} (E_r \oplus \bm{\sigma}_r)$. For $\bm{b} \in \overline{\mathbb{Q}}_{2^m}^s$ and $J=J(\bm{b})$ we have
\begin{align*}
 \Delta(E, J)   & = \frac{1}{2^m} \sum_{r=1}^{2^m} \frac{1}{2^m} \sum_{n=0}^{2^m-1} \bm{1}_J(\bm{x}_n(r) \oplus \bm{\sigma}_r) - \lambda(J)  \\ & = \frac{1}{2^m} \sum_{r=1}^{2^m} \frac{1}{2^m} \sum_{n=0}^{2^m-1} \left( \bm{1}_J(\bm{x}_n(r) \oplus \bm{\sigma}_r) - \mathbb{E}_{\bm{\sigma}_r}[\bm{1}_{J}(\bm{x}_n(r) \oplus \bm{\sigma}_r)] \right),
\end{align*}
where we used Lemma~\ref{lem_Q}.

In \cite[Proof of Theorem~1]{HNWW01} (see also \cite[Proposition~3.17]{DKP10}) it is shown that the maximum of $|\Delta(E, J(\bm{b}))|$ over all $\bm{b} \in \overline{\mathbb{Q}}_{2^m}^s$ differs at most by $s/ 2^m$ from the star discrepancy of $E$. In the present case, this implies 
\begin{equation}\label{P_dis2}
D_{2^{2m}}^\ast(E) \le \max_{\bm{b} \in \overline{\mathbb{Q}}_{2^m}^s} \left| \frac{1}{2^m} \sum_{r=1}^{2^m} \frac{1}{2^m} \sum_{n=0}^{2^m-1} \left( \bm{1}_J(\bm{x}_n(r) \oplus \bm{\sigma}_r) - \mathbb{E}_{\bm{\sigma}_r}[\bm{1}_{J}(\bm{x}_n(r) \oplus \bm{\sigma})] \right) \right|  + \frac{s}{2^m}.
\end{equation}

\subsection{The proof of Theorem~\ref{thm1}}\label{sec:main}

We now use the point set $P = \bigcup_{r=1}^{2^m} ( P_p(q_1) \oplus \bm{\sigma}_r)$, where $q_1, \ldots, q_{2^m} \in G_m$ and $\bm{\sigma}_1, \ldots, \bm{\sigma}_{2^m} \in \mathbb{Q}_{2^m}^s$ are chosen i.i.d. uniformly distributed. Since we count points according to their multiplicity, we have $|P| = 2^{2^m} =: N$.
Let
\begin{equation}\label{defYr}
Y_r(J):=  \frac{1}{2^m} \sum_{n=0}^{2^m-1} \left(\mathbf{1}_J(\bm{x}_n(q_r) \oplus \bm{\sigma}_r) -  \mathbb{E}_{q, \bm{\sigma}} [ \mathbf{1}_J( \bm{x}_n(q) \oplus \bm{\sigma}) ] \right).
\end{equation}
Lemma~\ref{lem_Q} implies that $\mathbb{E}_{\bm{\sigma}_r}(Y_r(J)) = 0$. From \eqref{P_dis2} we have
\begin{equation}\label{bd:from2.3}
D_{N}^\ast(P) \le \max_{\bm{b} \in \overline{\mathbb{Q}}_{2^m}^s} \left|\frac{1}{2^m} \sum_{r=1}^{2^m} Y_r(J) \right| + \frac{s}{2^m}
\end{equation}
In the following we will use Bernstein's inequality to show that one can find suitable $\{(q_1, \bm{\sigma}_1), \ldots, (q_{2^m}, \bm{\sigma}_{2^m})\}$ with high probability such that the discrepancy of $P$ is small. To be able to apply Bernstein's inequality, we need a bound on the variance of $Y_r(J)$.

\paragraph{Local discrepancy variance estimate.} The next lemma is needed in order to estimate the variance of $Y_r(J)$.

\begin{lemma}\label{lem_EDelta}
Let $m$ be a natural number and $p \in \mathbb{Z}_2[x]$ be an irreducible polynomial of degree $m$. Choose $q \in G_m$ and $\bm{\sigma} \in \mathbb{Q}_{2^m}^s$ i.i.d. uniformly distributed. Then for  $\bm{b} \in  \overline{\mathbb{Q}}_{2^m}^s$ we have
\begin{equation*}
\mathbb{E}_{q, \bm{\sigma}} [\Delta^2(P_p(q) \oplus \bm{\sigma}, J(\bm{b}))] = \frac{s}{2^m} \prod_{j=1}^s b_j \left( 1 - \prod_{j=1}^s b_j \right) \le \frac{s}{2^m}.
\end{equation*}
\end{lemma}

\begin{proof}
We have
\begin{align*}
& \mathbb{E}_{q, \bm{\sigma}}[ \Delta^2(P_p(q) \oplus \bm{\sigma}, J(\bm{b}))] \\ 
& = \frac{1}{2^{m}}  \sum_{q \in G_m} \frac{1}{2^{ms}} \sum_{\bm{\sigma} \in \mathbb{Q}_{2^m}^s} \Delta^2(P_p(q) \oplus \bm{\sigma},J(\bm{b})) \\ 
& =  \sum_{\bm{k}, \bm{k}' \in \{0, \ldots, 2^{m}-1\}^s \setminus \{\bm{0}\}}  c_{\bm{k}} c_{\bm{k}'} \frac{1}{2^{2m}} \sum_{n, n' =0}^{2^m-1} \frac{1}{2^{m}} \sum_{q \in G_m}  \operatorname{wal}_{\bm{k}}(\bm{x}_n(q)) \operatorname{wal}_{\bm{k}'}(\bm{x}_{n'}(q)) \frac{1}{2^{ms}} \sum_{\bm{\sigma} \in \mathbb{Q}_{2^m}^s} \operatorname{wal}_{\bm{k} \oplus \bm{k}'}(\bm{\sigma}) \\ 
& =  \sum_{\bm{k} \in \{0, \ldots, 2^{m}-1\}^s \setminus \{\bm{0}\}}  | c_{\bm{k}} |^2  \frac{1}{2^{2m}} 
\sum_{n, n'=0}^{2^m-1} \frac{1}{2^{m}} \sum_{q \in G_m} \operatorname{wal}_{\bm{k}}(\bm{x}_n(q) \oplus \bm{x}_{n'}(q)),
\end{align*}
where we used \eqref{fo:Delta1} and the fact that $\sum_{\bm{\sigma} \in \mathbb{Q}_{2^m}^s} \operatorname{wal}_{\bm{k} \oplus \bm{k}'}(\bm{\sigma})=0$ whenever $\bm{k} \not=\bm{k}'$ and $2^{ms}$ otherwise. 

We have $\bm{x}_n(q) \oplus \bm{x}_{n'}(q) = \bm{x}_{\ell}(q)$ for $\ell \equiv n \oplus n' \pmod{p}$ (when we identify non-negative integers with polynomials in $\mathbb{Z}_2[x]$ in the natural way). Thus
\begin{equation*}
\mathbb{E}_{q, \bm{\sigma}}[\Delta^2(P_p(q) \oplus \bm{\sigma}, J(\bm{b}))] =  \sum_{\bm{k} \in \{0, \ldots, 2^{m}-1\}^s \setminus \{\bm{0}\}}  | c_{\bm{k}} |^2  \frac{1}{2^{m}} 
\sum_{\ell =0}^{2^m-1} \frac{1}{2^{m}} \sum_{q \in G_m} \operatorname{wal}_{\bm{k}}(\bm{x}_\ell(q)).
\end{equation*}

We have
\begin{equation*}
\left| \frac{1}{2^m} \sum_{\ell=0}^{2^m-1} \frac{1}{2^m} \sum_{q \in G_m} \operatorname{wal}_{\bm{k}}(\bm{x}_\ell(q)) \right| = \left| \frac{1}{2^m} \sum_{q \in G_m} 1\{ \bm{k} \cdot (1, q, \ldots, q^{s-1}) \equiv 0 \pmod{p} \} \right| \le \frac{s-1}{2^m},
\end{equation*}
where in the last step we used the fundamental theorem of algebra. Therefore
\begin{align*}
\mathbb{E}_{q, \bm{\sigma}}[\Delta^2(P_p(q) \oplus \bm{\sigma}, J(\bm{b}))] & \le  \frac{s-1}{2^m} \sum_{\bm{k} \in \{0, \ldots, 2^{m}-1\}^s \setminus \{\bm{0}\}}  | c_{\bm{k}} |^2   \le \frac{s}{2^m} \prod_{j=1}^s b_j  \left( 1 - \prod_{j=1}^s b_j \right),
\end{align*}
where we used  Lemma~\ref{lem_aux}.
\end{proof}

Since $\mathbb{E}[Y_r(J)] = 0$, we have
\begin{equation*}
\operatorname{Var}[Y_r(J)] = \mathbb{E}[Y_r^2(J)] = \mathbb{E}_{q, \bm{\sigma}}[\Delta^2(P_p(q) \oplus \bm{\sigma}, J(\bm{b}))] \le \frac{s}{2^m}
\end{equation*}
by Lemma~\ref{lem_EDelta}.

In summary, we have the following properties. For fixed $J=J(\bm{b})$ with $\bm{b} \in (\mathbb{Q}_{2^m} \cup \{1\})^s$, the random variables $Y_r(J)$, $r=1,,\ldots,2^m$, are i.i.d. and satisfy
\begin{equation}\label{cond_Yr}
\begin{aligned}
|Y_r(J)| & \le 1, \\ \mathbb{E}[ Y_r(J)] & = 0, 
\\ \operatorname{Var}\left[ \sum_{r=1}^{2^m} Y_r(J) \right]  = \sum_{r=1}^{2^m} \operatorname{Var}[Y_r(J)] & \le s.
\end{aligned}
\end{equation}

\paragraph{Applying Bernstein's inequality.} In order to take advantage of the small variance due to using lattice point sets, we use Bernstein's inequality \cite{ber1946} (rather than Hoeffding's inequality).

Let $X_1,\dots, X_R$ be independent real-valued random variables with $
\mathbb{E}[X_r] = 0$ and $|X_r | \le c$ almost surely for all $r$ for some $c > 0$. Set
\[
  S_R  \coloneqq \sum_{r=1}^R X_r
  \qquad \text{ and } \qquad  v \coloneqq \sum_{r=1}^R \operatorname{Var}[X_r].
\]
Then, for every $t \ge 0$,
\begin{equation*}
\mathbb{P}\left[ |S_R| \ge t \right] \le 2 \exp\left( - \frac{t^2}{2 (v+ct /3)} \right).
\end{equation*}

We apply Bernstein's inequality with $X_r=Y_r(J)$, where $J = [\bm{0}, \bm{b})$, $\bm{b} \in \overline{\mathbb{Q}}_{2^m}^s$ is given, and $R=2^m$. In this case we obtain from \eqref{cond_Yr} that \[v \le s.\]
Hence
\begin{equation*}
\mathbb{P}\left[|S_{2^m} | >t\right]
  \;\le\;
  2\exp\!\left(
    -\,\frac{t^2}{ 2s + 2 t /3}\right).
\end{equation*}

\paragraph{Existence result.} 
Note that $|\overline{\mathbb{Q}}_{2^m}^s|=(2^m+1)^s$. Using a union bound argument we obtain 
\begin{equation*}
\mathbb{P}\left[\exists \bm{b} \in \overline{\mathbb{Q}}_{2^m}^s : |S_{2^m}(J(\bm{b}))| > t \right] \le  2 (2^m+1)^s \exp\left( -\frac{t^2}{2s + 2 t / 3}  \right),
\end{equation*}
and therefore
\begin{equation*}
\mathbb{P}\left[\forall \bm{b} \in \overline{\mathbb{Q}}_{2^m}^s: |S_{2^m}(J(\bm{b}))| \le t \right] \ge 1 -  2 (2^m+1)^s \exp\left( - \frac{t^2}{2s + 2 t /3} \right).
\end{equation*}

Let $\delta \in (0, 1)$. We want to find the smallest $t > 0$ such that
\begin{equation*}
2 (2^m+1)^s \exp \left( - \frac{t^2}{2s + 2 t/3}   \right) \le 1-\delta,
\end{equation*}
which is equivalent to
\begin{equation*}
t^2 - \left(2s + \frac{2 t}{3}\right)  (\log(2 (2^m+1)^s) - \log(1-\delta))  \ge 0.
\end{equation*}
Let $t_0$ be such that we get equality, i.e.,  
\begin{equation*}
t_0^2 -  t_0 \frac{2 (\log(2 (2^m+1)^s) - \log(1-\delta)) }{3}  -  2 s  ( \log(2 (2^m+1)^s) - \log(1-\delta) ) =  0.
\end{equation*}
Solving the quadratic equation for $t_0$ and selecting the positive solution (the other solution is negative) we obtain
\begin{align*}
t_0  = \frac{\log(2 (2^m+1)^s) - \log(1-\delta)}{3} \left(1 + \sqrt{1 +  \frac{ 18 s }{\log(2 (2^m+1)^s) - \log(1-\delta) }}  \right).
\end{align*}
Then we have
\begin{equation*}
\mathbb{P}\left[ \forall \bm{b} \in \overline{\mathbb{Q}}_{2^m}^s : \left| \frac{1}{2^m} \sum_{r=1}^{2^m} Y_r(J(\bm{b})) \right| \le \frac{t_0 }{2^m} \right] \ge \delta>0.
\end{equation*}
Thus with probability at least $\delta$ we have for all intervals $J(\bm{b})$, $\bm{b} \in \overline{\mathbb{Q}}_{2^m}^s$, 
\begin{align}
& \left| \frac{1}{2^m}  \sum_{r=1}^{2^m} Y_r(J(\bm{b})) \right| \nonumber \\ & \le \frac{\log(2 (2^m+1)^s) - \log(1-\delta)}{3 \cdot 2^m} \left(1   + \sqrt{1 + \frac{18 s}{\log(2 (2^m+1)^s) - \log(1-\delta) } } \right) \nonumber \\
& \le \frac{\log(2^{(m+1)s})+\log 2 - \log(1-\delta)}{2^m} \underbrace{\frac{1}{3}\left(1   + \sqrt{1 + \frac{18}{\log 3} } \right)}_{=1.723\ldots} \label{Yr_var_bound2}
\end{align}

The total number of points of $P$ is $N = 2^{2m}$. Thus we obtain with probability at least $\delta \in (0,1)$ that for all intervals $J(\bm{b})$, $\bm{b} \in \overline{\mathbb{Q}}_{2^m}^s$,
\begin{equation}\label{bound_ave_Yr}
\left| \frac{1}{2^m} \sum_{r=1}^{2^m} Y_r(J(\bm{b})) \right| \le (1.723\ldots) \times  \frac{s \log(2 N) + \log 2 - \log(1-\delta)}{N^{1/2}}.
\end{equation}

From \eqref{bd:from2.3} and \eqref{bound_ave_Yr} we obtain that for any $\delta \in (0,1)$, with probability at least $\delta$ that the point set $P \subseteq \mathbb{Q}_{2^m}^s$ from \eqref{defP} with $N = 2^{2m}$ points, satisfies
\begin{equation*}
D_N^\ast(P) \le (1.723\ldots) \times  \frac{s (\log(2 N) +1)+ \log 2 - \log(1-\delta)}{N^{1/2}}.
\end{equation*}
This finishes the proof of Theorem~\ref{thm1}.

\subsection{The proof of Theorem~\ref{thm2}}\label{sec:main2}

Consider first generic point sets $P_1, P_2, \ldots, P_{2^m} \subseteq [0,1)^s$, with $P_r = \{\bm{x}_0(r), \ldots, \bm{x}_{2^m-1}(r)\}$. Let $\bm{\sigma}_1, \ldots, \bm{\sigma}_{2^m} \in \mathbb{Q}_{2^m}^s$ be i.i.d. uniformly distributed. Define
\begin{equation*}
X_r(J) := \Delta(P_r \oplus \bm{\sigma}_r, J) = \frac{1}{2^m} \sum_{n=0}^{2^m-1} \bm{1}_J(\bm{x}_n(r) \oplus \bm{\sigma}_{r}) - \lambda(J).
\end{equation*}

Then $|X_r(J)| \le 1$ and due to the random digital shift, for $J=J(\bm{b})$, $\bm{b} \in \overline{\mathbb{Q}}_{2^m}^s$, by Lemma~\ref{lem_Q} we have  
\begin{align*}
\mathbb{E}_{\bm{\sigma}_{r}}[X_r(J)] & = \frac{1}{2^m} \sum_{n=0}^{2^m-1} \mathbb{E}_{\bm{\sigma}_{r}} [\bm{1}_J(\bm{x}_n(r) \oplus \bm{\sigma}_{r})] - \prod_{j=1}^s b_j  = \mathbb{E}_{\bm{\sigma}_{r}} [\bm{1}_{J}(\bm{\sigma}_{r}) ] - \prod_{j=1}^s b_j = 0.
\end{align*}
Let 
\begin{equation*}
T_{2^m}(J) := \sum_{r=1}^{2^m} X_r(J).
\end{equation*}
Then $\mathbb{E}_{\bm{\sigma}_1, \ldots, \bm{\sigma}_{2^m}}[T_{2^m}(J)] = \sum_{r=1}^{2^m} \mathbb{E}_{\bm{\sigma}_{r}}[ X_r(J) ] = 0$. In the following we estimate $\operatorname{Var}[T_{2^m}(J)] = \mathbb{E}_{\bm{\sigma}_1, \ldots, \bm{\sigma}_{2^m}}[ T_{2^m}^2(J)]$.

\begin{lemma}\label{lem_EDeltaP2}
Let $P_1, \ldots, P_{2^m} \subseteq [0,1)^s$ be point sets with $P_r = \{\bm{x}_0(r), \ldots, \bm{x}_{2^m-1}(r)\}$. 
Assume that for all $\bm{k} \in \{0, \ldots, 2^{m-1}\}^s \setminus \{\bm{0}\}$ we have
\begin{equation}\label{assume_PB}
\left| \sum_{r=1}^{2^m} \frac{1}{2^{2m}} \sum_{n, n'=0}^{2^m-1} \operatorname{wal}_{\bm{k}}(\bm{x}_n(r) \oplus \bm{x}_{n'}(r)) \right| \le B,
\end{equation}
for some constant $B$ independent of $\bm{k}$.

Choose  $\bm{\sigma}_1, \ldots, \bm{\sigma}_{2^m} \in \mathbb{Q}_{2^m}^s$ i.i.d. uniformly distributed. Then for  $\bm{b} \in \overline{\mathbb{Q}}_{2^m}^s$ we have
\begin{equation*}
\mathbb{E}_{\bm{\sigma}_1, \ldots, \bm{\sigma}_{2^m}} [ T_{2^m}^2(J(\bm{b}))] \le B \prod_{j=1}^s b_j \left( 1 - \prod_{j=1}^s b_j \right) \le B.
\end{equation*}
\end{lemma}

\begin{proof}
We have
\begin{align*}
\mathbb{E}_{\bm{\sigma}_1, \ldots, \bm{\sigma}_{2^m}}[ T_{2^m}^2(J(\bm{b}))]  
& = \sum_{r, r'=0}^{2^m-1} \mathbb{E}_{\bm{\sigma}_1, \ldots, \bm{\sigma}_{2^m}} [ X_r(J) X_{r'}(J) ].
\end{align*}
For $r \neq r'$ we have $\mathbb{E}[X_r(J) X_{r'}(J)] = \mathbb{E}[X_r(J)] \mathbb{E}[X_{r'}(J)] = 0$. Hence, similarly as in the proof of Lemma~\ref{lem_EDelta},
\begin{align*}
& \mathbb{E}_{\bm{\sigma}_1, \ldots, \bm{\sigma}_{2^m}}[ T_{2^m}^2(J(\bm{b}))]   = \sum_{r =1}^{2^m } \mathbb{E}_{ \bm{\sigma}_{r }} [ X^2_r(J) ]  =  \sum_{r=1}^{2^m} \mathbb{E}_{\bm{\sigma}_{r}} [ \Delta^2( P_r \oplus \bm{\sigma}_r, J)] \\ & = \sum_{r=1}^{2^m} \sum_{\bm{k}, \bm{k}' \in \{0, \ldots, 2^m-1\}^s \setminus \{\bm{0}\}} c_{\bm{k}} c_{\bm{k}'} \frac{1}{2^{2m}} \sum_{n, n' = 0}^{2^m-1} \operatorname{wal}_{\bm{k}}(\bm{x}_n(r)) \operatorname{wal}_{\bm{k}'}(\bm{x}_{n'}(r)) \frac{1}{2^{ms}} \sum_{\bm{\sigma}_{r+1} \in \mathbb{Q}_{2^m}^s} \operatorname{wal}_{\bm{k} \oplus \bm{k}'}(\bm{\sigma}_{r}) \\ & =  \sum_{\bm{k} \in \{0, \ldots, 2^m-1\}^s \setminus \{\bm{0}\}} | c_{\bm{k}} |^2 \sum_{r=1}^{2^m} \frac{1}{2^{2m}} \sum_{n, n' = 0}^{2^m-1} \operatorname{wal}_{\bm{k}}(\bm{x}_n(r) \oplus \bm{x}_{n'}(r)). 
\end{align*}
The result now follows from \eqref{assume_PB}.
\end{proof}

In summary, we have the following properties. For fixed $J=J(\bm{b})$, $\bm{b} \in \overline{\mathbb{Q}}_{2^m}^s$ the random variables $X_r(J)$, $r=1,\ldots,2^m$, are i.i.d. and satisfy
\begin{align*}
|X_r(J)| & \le 1, \\ \mathbb{E}[ X_r(J)] & = 0, \\   \operatorname{Var} \left[ \sum_{r=0}^{2^m-1} X_r(J) \right] & \le B.
\end{align*}
Thus the random variables $X_r(J)$ satisfy the similar properties as $Y_r(J)$ given in \eqref{cond_Yr}. Thus the results from the proof of Theorem~\ref{thm1} apply accordingly. In particular, the bound \eqref{bound_ave_Yr} applies also to $2^{-m} \sum_{r=1}^{2^m} X_r(J)$:  With probability at least $\delta \in (0,1)$ for all $J$ it is true that
\begin{equation*}
\left| \frac{1}{2^m} \sum_{r=1}^{2^m} X_r(J) \right| \le \frac{\log(2 (2^m+1)^s) - \log(1-\delta)}{3 \cdot 2^m} \left(1   + \sqrt{1 + \frac{18 B}{\log(2 (2^m+1)^s) - \log(1-\delta) } } \right).
\end{equation*}

From \eqref{P_dis2} and Lemma~\ref{lem_EDeltaP2} we obtain the following lemma.

\begin{lemma}\label{lem_Pgen}
Let $P_1, \ldots, P_{2^m} \subseteq [0,1)^s$ be point sets with $P_r = \{\bm{x}_0(r), \ldots, \bm{x}_{2^m-1}(r)\}$. Assume that for all $\bm{k} \in \{0, \ldots, 2^{m-1}\}^s \setminus \{\bm{0}\}$ we have
\begin{equation*}
\left| \sum_{r=1}^{2^m} \frac{1}{2^{2m}} \sum_{n, n'=0}^{2^m-1} \operatorname{wal}_{\bm{k}}(\bm{x}_n(r) \oplus \bm{x}_{n'}(r)) \right| \le B,
\end{equation*}
for some constant $B$ independent of $\bm{k}$.

Choose  $\bm{\sigma}_1, \ldots, \bm{\sigma}_{2^m} \in \mathbb{Q}_{2^m}^s$ i.i.d. uniformly distributed. Let
\begin{equation*}
Q = \bigcup_{r=1}^{2^m} (P_r \oplus \bm{\sigma}_r).
\end{equation*}
Then for any $\delta \in (0,1)$, with probability at least $\delta$ that the point set  $Q \subseteq [0,1)^s$ with $N = 2^{2 m}$ points, satisfies
\begin{align*}
D_N^\ast(Q) \le  \frac{\log(2 (2^m+1)^s) - \log(1-\delta)}{3 \cdot 2^m} \left(1   + \sqrt{1 + \frac{18 B}{\log(2 (2^m+1)^s) - \log(1-\delta) } } \right) + \frac{s}{2^m}.
\end{align*}
\end{lemma}

It remains to show the sets $P_p(0), \ldots, P_p(2^m-1)$ satisfy \eqref{assume_PB} with $B = s$.
\begin{lemma}\label{lem_Kp}
Let $m$ be a natural number and $p \in \mathbb{Z}_2[x]$ be an irreducible polynomial of degree $m$. Let $P_p(r) = \{ \bm{x}_0(r), \ldots, \bm{x}_{2^m-1}(r) \}$ for $r = 0, 1, \ldots, 2^m-1$. Then
\begin{equation*}
\left| \sum_{r=0}^{2^m-1} \frac{1}{2^{2m}} \sum_{n, n'=0}^{2^m-1} \operatorname{wal}(\bm{x}_n(r) \oplus \bm{x}_{n'}(r)) \right| \le s.
\end{equation*}
\end{lemma}

\begin{proof}
We have $\bm{x}_n(r) \oplus \bm{x}_{n'}(r) = \bm{x}_{\ell}(r)$ for $\ell \equiv n \oplus n' \pmod{p}$ (when we identify non-negative integers with polynomials in $\mathbb{Z}_2[x]$ in the natural way). Thus
\begin{align*}
\sum_{r=0}^{2^m-1} \frac{1}{2^{2m}} \sum_{n, n'=0}^{2^m-1} \operatorname{wal}_{\bm{k}}(\bm{x}_n(r) \oplus \bm{x}_{n'}(r))  & =   \sum_{r=0}^{2^m-1} \frac{1}{2^{m}} \sum_{\ell = 0}^{2^m-1} \operatorname{wal}_{\bm{k}}(\bm{x}_\ell(r) ) \\  & =   \sum_{r=0}^{2^m-1} 1\{\bm{k} \cdot (1, r, \ldots, r^{s-1}) \equiv 0 \pmod{p} \} \\ & \le s-1,
\end{align*}
where the last inequality follows from the fact that the polynomial $k_1 + k_2 r + \cdots + k_s r^{s-1} \equiv 0\pmod{p}$ in the variable $r$ has at most $s-1$ solutions.
\end{proof}

Using Lemmas~\ref{lem_Pgen} and \ref{lem_Kp} and the estimation \eqref{Yr_var_bound2} finishes the proof of Theorem~\ref{thm2}.

\section*{Acknowledgement}

The study of the dimensional dependence of the inverse star discrepancy was first proposed by G. Larcher, as noted in \cite{HNWW01}. One of Gerhard's interest has always been to find explicit constructions of point sets attaining these bounds. In this paper we make progress in that direction, and so we dedicate this work to him on the occasion of his retirement.\\

We thank two anonymous referees for their helpful comments, in particular for suggesting the content of Remark~\ref{re1.1}.


\begin{thebibliography}{99}

\bibitem{Ai11}
C.~Aistleitner.
\newblock Covering numbers, dyadic chaining and discrepancy.
\newblock {\em Journal of Complexity} \textbf{27} (2011), no. 6, 531--540.

\bibitem{Aistleitner2014weights}
C.~Aistleitner.
\newblock Tractability results for the weighted star-discrepancy.
\newblock {\em Journal of Complexity} \textbf{30} (2014), 381--391.

\bibitem{AD15}
C.~Aistleitner, J.~Dick.
\newblock Functions of bounded variation, signed measures, and a general Koksma--Hlawka inequality.
\newblock {\em Acta Arithmetica} \textbf{167} (2015), 143--171.



\bibitem{AH14}
C.~Aistleitner, M.~Hofer.
\newblock Probabilistic discrepancy bound for {M}onte {C}arlo point sets.
\newblock {\em Mathematics of Computation} \textbf{83} (2014), 1373--1381.



\bibitem{Baldeaux2011}
J.~Baldeaux, J.~Dick, J.~Greslehner, F.~Pillichshammer.
\newblock Construction algorithms for higher order polynomial lattice rules.
\newblock {\em Journal of Complexity} \textbf{27} (2011), 281--299.

\bibitem{ber1946}
S.N.~Bernstein.
\newblock {\em The Theory of Probabilities.} 
\newblock Gastehizdat Publishing House, Moscow, 1946.



\bibitem{Clément22JCO}
F.~Cl\'ement, C.~Doerr, L.~Paquete.
\newblock Star discrepancy subset selection: problem formulation and efficient approaches for low dimensions.
\newblock {\em Journal of Complexity} \textbf{73} (2022), Paper No. 101645.


\bibitem{Clément23GECCO}
F.~Cl\'ement, D.~Vermetten, J.~de~Nobel, A.~D.~Jesus, L.~Paquete, C.~Doerr.
\newblock Computing star discrepancies with numerical black-box optimization algorithms.
\newblock In {\em Proc.\ GECCO 2023}, 1330--1338 (2023).

\bibitem{D14}
J.~Dick.
\newblock Numerical integration of Hölder continuous, absolutely convergent Fourier, Fourier cosine, and Walsh series.
\newblock {\em Journal of Approximation Theory} \textbf{183} (2014), 14--30.

\bibitem{DGS25}
J.~Dick, T.~Goda, K.~Suzuki.
\newblock Tractability results for integration in subspaces of the Wiener algebra.
\newblock {\em Journal of Complexity} \textbf{90} (2025), Paper No. 101948.

\bibitem{DKLP07}
J.~Dick, P.~Kritzer, G.~Leobacher, F.~Pillichshammer.
\newblock Constructions of general polynomial lattice rules based on the weighted star discrepancy.
\newblock {\em Finite Fields and Their Applications} \textbf{13} (2007), 1045--1070.


\bibitem{DKP22}
J.~Dick, P.~Kritzer, F.~Pillichshammer. 
\newblock {\em Lattice Rules - Numerical Integration, Approximation, and Discrepancy}. 
\newblock Springer, Cham, 2022.

\bibitem{DKS13}
J.~Dick, F.~Y.~Kuo, I.~H.~Sloan.
\newblock High-dimensional integration: the quasi--Monte Carlo way.
\newblock {\em Acta Numerica} \textbf{22} (2013), 133--288.



\bibitem{DKLS05}
J.~Dick, F.~Y.~Kuo, F.~Pillichshammer, I.~H.~Sloan.
\newblock Construction algorithms for polynomial lattice rules for multivariate integration.
\newblock {\em Mathematics of Computation} \textbf{74} (2005), 1895--1921.



\bibitem{DKP10}
J.~Dick and F.~Pillichshammer.
\newblock {\em Digital Nets and Sequences: Discrepancy Theory and Quasi--Monte Carlo Integration}.
\newblock Cambridge University Press, 2010.

\bibitem{Doerr14}
B.~Doerr.
\newblock A lower bound for the discrepancy of a random point set.
\newblock {\em Journal of Complexity} \textbf{30} (2014), 16--20.

\bibitem{DoerrJittered}
B.~Doerr.
\newblock A sharp discrepancy bound for jittered sampling.
\newblock {\em Mathematics of Computation} \textbf{91} (2022), 1871--1892.

\bibitem{DGKP08}
B.~Doerr, M.~Gnewuch, P.~Kritzer, F.~Pillichshammer.
\newblock Component-by-component construction of low-discrepancy point sets of small size.
\newblock {\em Monte Carlo Methods and Applications} \textbf{14} (2008), no. 2, 129--149.

\bibitem{DGW10}
B.~Doerr, M.~Gnewuch, M. Wahlstr\"om.
\newblock Algorithmic construction of low discrepancy point sets via dependent randomized rounding.
\newblock {\em Journal of Complexity} \textbf{26} (2010), 490--507. 

\bibitem{Fine} 
N.~Fine.
\newblock On the Walsh functions. 
\newblock {\em Transactions of the American Mathematical Society} \textbf{65} (1949), 372--414.

\bibitem{GPW} 
M.~Gnewuch, H.~Pasing, Ch.~Weiss. 
\newblock A generalized Faulhaber inequality, improved bracketing covers, and applications to discrepancy. 
\newblock {\em Mathematics of Computation} \textbf{90} (2021), no. 332, 2873--2898. 

\bibitem{GSW09}
M.~Gnewuch, A.~Srivastav, and C.~Winzen.
\newblock Finding optimal volume subintervals with $k$ points and calculating the star discrepancy are NP-hard problems.
\newblock {\em Journal of Complexity} \textbf{25} (2009), 115--127.

\bibitem{GWW12}
M.~Gnewuch, M.~Wahlstr\"om, and C.~Winzen.
\newblock A new randomized algorithm to approximate the star discrepancy based on threshold accepting.
\newblock {\em SIAM Journal on Numerical Analysis} \textbf{50} (2012), 781--807.

\bibitem{G23}
T.~Goda.
\newblock Polynomial tractability for integration in an unweighted function space with absolutely convergent Fourier series
\newblock {\em Proceedings of the American Mathematical Society} \textbf{151} (2023), no. 9, 3925--3933.

\bibitem{GIWV20}
D.~Gross, M.~Iwen, L.~K\"ammerer, and T.~Volkmer.
\newblock A deterministic algorithm for constructing multiple rank-1 lattices of near-optimal size.
\newblock {\em Advances in Computational Mathematics} \textbf{47} (2021), No. 86.


\bibitem{HNWW01}
S.~Heinrich, E.~Novak, G.~W.~Wasilkowski, and H.~Wo\'zniakowski.
\newblock The inverse of the star-discrepancy depends linearly on the dimension.
\newblock {\em Acta Arithmetica} \textbf{96} (2001), 279--302.

\bibitem{Hin04}
A.~Hinrichs.
\newblock Covering numbers, Vapnik--{\v C}ervonenkis classes and bounds for the star-discrepancy.
\newblock {\em Journal of Complexity} \textbf{20} (2004), 477--483.


\bibitem{HPS08}
A.~Hinrichs, F.~Pillichshammer, and W.~Ch.~Schmid.
\newblock Tractability properties of the weighted star discrepancy.
\newblock {\em Journal of Complexity} \textbf{24} (2008), 134--143.


\bibitem{HinrichsMarkhasinHalton}
A.~Hinrichs, F.~Pillichshammer, S.~Tezuka.
\newblock Tractability properties of the weighted star discrepancy of the Halton sequence.
\newblock {\em Journal of Computational and Applied Mathematics} \textbf{350} (2019), 46--54.



\bibitem{Kaemmerer18}
L.~K\"ammerer.
\newblock Multiple rank-1 lattices as sampling schemes for multivariate trigonometric polynomials.
\newblock {\em Journal of Fourier Analysis and Applications} \textbf{24} (2018), 17--44.

\bibitem{KPV21}
L.~K\"ammerer, D.~Potts, T.~Volkmer.
\newblock High-dimensional sparse FFT based on sampling along multiple rank-1 lattices.
\newblock {\em Applied Computational Harmonic Analysis} \textbf{51} (2021), 225-257.




\bibitem{Lobbe14}
T.~L\"obbe.
\newblock Probabilistic star discrepancy bounds for lacunary point sets.
\newblock {\em arXiv:1408.2220} (2014).




\bibitem{NP}
M.~Neum\"uller, F.~Pillichshammer. 
\newblock Metrical star discrepancy bounds for lacunary subsequences of digital Kronecker-sequences and polynomial tractability. 
\newblock {\em Uniform Distribution Theory} \textbf{13} (2018), no.~1, 65--86. 


\bibitem{niesiam}
H.~Niederreiter.
\newblock {\em Random Number Generation and Quasi-Monte Carlo Methods},
  volume~63 of {\em CBMS-NSF Series in Applied Mathematics}.
\newblock SIAM, Philadelphia, 1992.

\bibitem{NW08} 
E. Novak, H. Wo\'{z}niakowski.
\newblock Tractability of Multivariate Problems. Volume I: Linear Information. 
\newblock {\em EMS Tracts in Mathematics}, 16, Z\"urich, 2008.

\bibitem{NW10} 
E. Novak, H. Wo\'{z}niakowski. 
\newblock Tractability of Multivariate Problems. Volume II: Standard Information for Functionals. \newblock {\em EMS Tracts in Mathematics}, 12, Z\"urich, 2010.

\bibitem{NC06a}
D.~Nuyens, R.~Cools.
\newblock Fast component-by-component construction of rank-1 lattice rules with a non-prime number of points.
\newblock {\em Journal of Complexity} \textbf{22} (2006), 4--28.

\bibitem{NC06b}
D.~Nuyens, R.~Cools.
\newblock Fast algorithms for component-by-component construction of rank-1 lattice rules in shift-invariant reproducing kernel Hilbert spaces.
\newblock {\em Mathematics of Computation} \textbf{75} (2006), 903--920.





\bibitem{SWS} 
F.~Schipp, W.R.~Wade, P.~Simon. 
\newblock {\em Walsh series. An introduction to dyadic harmonic analysis.} 
\newblock Adam Hilger, Ltd., Bristol, 1990.


\bibitem{Steinerberger22}
S.~Steinerberger.
\newblock An elementary proof of a lower bound for the inverse of the star discrepancy.
\newblock {\em Journal of Complexity} \textbf{75} (2022), Paper No. 101713.

\bibitem{Wei26}
C.~Wei{\ss}.
\newblock Hammersley point sets and inverse of star-discrepancy.
\newblock {\em Journal of Complexity} \textbf{93} (2026), Paper No. 101998.

\bibitem{WGH20}
M.~Wnuk, M.~Gnewuch, N.~Hebbinghaus.
\newblock On negatively dependent sampling schemes, variance reduction, and probabilistic upper discrepancy bounds.
\newblock Radon Ser. Comput. Appl. Math., 26 De Gruyter, Berlin, 2020, 43--67.

\end{thebibliography}
\end{document}